\documentclass[11pt]{article}
\usepackage[utf8]{inputenc}
\usepackage{amsmath, amssymb, amsthm, enumerate}
\usepackage{hyperref}	
\usepackage[pdftex]{graphicx}
\usepackage{asymptote}
\usepackage{indentfirst}
\usepackage{mathrsfs}
\usepackage[left=2cm,right=2cm,top=2cm,bottom=2cm]{geometry}

\usepackage[capitalize,nameinlink]{cleveref}
\crefname{theorem}{Theorem}{Theorems}
\crefname{lemma}{Lemma}{Lemmas}
\crefname{claim}{Claim}{Claims}
\crefname{prop}{Proposition}{Propositions}

\usepackage{tikz-cd}
\usepackage{todonotes}
\usepackage{verbatim}

\usepackage{lineno,hyperref}
\modulolinenumbers[5]

\makeatletter
\newtheorem*{rep@theorem}{\rep@title}
\newcommand{\newreptheorem}[2]{%
\newenvironment{rep#1}[1]{%
 \def\rep@title{#2~\ref{##1}}%
 \begin{rep@theorem}}%
 {\end{rep@theorem}}}
\makeatother

\theoremstyle{plain}
\newtheorem{theorem}{Theorem}[section]
\newtheorem{lemma}[theorem]{Lemma}

\newreptheorem{thm}{Theorem}

\newtheorem{prop}[theorem]{Proposition}
\newreptheorem{prop}{Proposition}

\newreptheorem{cor}{Corollary}

\newreptheorem{conj}{Conjecture}

\theoremstyle{definition}
\newtheorem{definition}[theorem]{Definition}

\theoremstyle{definition}

\theoremstyle{remark}

\numberwithin{equation}{section}

\newcommand{\FF}{\mathbb{F}}
\newcommand{\F}{\mathbb{F}}

\newcommand{\NN}{\mathbb{N}}
\newcommand{\N}{\mathbb{N}}

\newcommand{\ZZ}{\mathbb{Z}}

\DeclareMathOperator{\Eval}{Eval}

\newcommand{\Esymb}{{\bf E}}
\DeclareMathOperator*{\E}{\Esymb}

\DeclareMathOperator{\Cov}{Cov}

\title{Iterated sumset expansion in $\mathbb{F}_p^n$}
\author{Manik Dhar \thanks{Massachusetts Institute of Technology. \texttt{dmanik@mit.edu}} \and Sammy Luo 
\thanks{Massachusetts Institute of Technology. Research supported by NSF Award No. 2303290. \texttt{sammyluo@mit.edu}}
}

\begin{document}

\maketitle

\begin{abstract}
    Given a set $A \subseteq \mathbb{F}_p^n$, what conditions does one need to guarantee that iterated sumsets of the form $A+\cdots+A$ expand quickly (say, within $O(p)$ terms) to the whole space? When only the size of $A$ is known, such expansion results are only possible when $|A|>\frac{1}{p}|\mathbb{F}_p^n|$. However, heuristic considerations suggest that expansion should begin with much smaller sets under just mild ``nondegeneracy'' conditions. In this paper, we confirm this intuition by showing a sufficient algebraic condition for the asymmetric version of this problem: We have $A_1+\dots+A_m=\mathbb{F}_p^n$ as long as each $A_i$ is not contained in the zero set of any low degree polynomial ($\deg = O(n)$ when $m=O(p)$). We close with a discussion of the behavior of random sets, as well as extensions of these results and connections with the Erd\H{o}s-Ginzburg-Ziv problem. Our proofs make use of the shift operator polynomial method developed by the second author.
\end{abstract}

\section{Introduction}

\label{sec:intro}

For subsets $A,B$ of an abelian group $G$, their sumset is defined by $A+B:=\{a+b:\: a\in A,\, b\in B\}$. This definition extends naturally to iterated sumsets of the form $A_1+\cdots+A_m$. Many of the central questions and results in additive combinatorics revolve around the study of the size of a sumset given information about the size or structure of the summands.

In this paper, we study the following question: Given subsets $A_1,\dots,A_m$ of a vector space $V=\FF^n$ over a finite field $\FF=\FF_q$, under what circumstances can we guarantee that
\[
A_1+\cdots+A_m=\FF_q^n,
\]
i.e. the sumset expands to the whole space?

A simple version of this question was posed by Adam Chapman on MathOverflow \cite{chapman2023}, in the case where $q=p$ is prime, $m=p-1$, $A:=A_1=\cdots=A_m$, and the only information assumed about $A$ is its size. An answer by Terry Tao points out a result of Bollob\'as and Leader \cite{BL1996} which implies that given the sizes of two sets $A$ and $B$ in $\F_p^n$, the size of their sumset is minimized when $A,B$ are initial segments in a lexicographic order on the coordinates. It follows that the condition $|A|\ge \frac{p^n-1}{p-1}+1$ is sufficient to imply $\underbrace{A+\cdots+A}_{p-1\text{ times}} = \FF_p^n$. This bound is tight, as seen from choosing $A$ to be the set of all points $(x_1,\dots,x_n)\in \FF_p^n$ whose first nonzero coordinate equals $1$.

The highly structured nature of this tight example, however, makes it natural to question whether the sumset expansion behavior we seek starts to show up for much smaller sets, as long as some modest structural constraints are satisfied. Our main result shows that this is indeed the case: over $\FF_p$ for $p$ prime, it suffices to have the condition that none of the sets $A_i$ lie in a low degree hypersurface, i.e. the zero set of a low degree polynomial.

\begin{theorem}
    \label{thm:main}
    Let $p$ be a prime, and let $m, n_1,\dots,n_m$ be positive integers such that $n_1+\cdots+n_m\ge (p-1)n$. If $A_1,\dots,A_m\subseteq \FF_p^n$, and for $1\le i\le m$, $A_i$ is not contained in the zero set of any polynomial of degree $\le n_i$, then 
    \[
        A_1+\cdots+A_{m} = \FF_p^n.
    \]
\end{theorem}

In particular, in the symmetric setting studied in \cite{chapman2023}, we have the following result.

\begin{theorem}
    \label{thm:main-symm}
    If $A\subseteq \FF_p^n$, and $A$ is not contained in the zero set of a polynomial of degree $\le n$, then 
    \[
        \underbrace{A+\cdots+A}_{p-1\text{ times}} = \FF_p^n.
    \]
\end{theorem}

For large $p$, there exist sets $A\subseteq \FF_p^n$ that do not lie in the zero set of any polynomial of degree $\le n$ but have size as small as $\binom{2n}{n}+1$, which is much smaller than $\frac{p^n-1}{p-1}+1$. Nevertheless, this does not seem to be quite the sharpest possible condition to impose; the tight example given with $|A|=\frac{p^n-1}{p-1}$ suggests that some more ``linear-looking'' constraint might be possible. Our result below confirms this intuition in the two-dimensional symmetric case.

\begin{theorem}
    \label{thm:2d}
    Let $p>2$ be a prime. If $A\subseteq \FF_p^2$ contains a set of $4$ points, no three of which are collinear, then 
    \[\underbrace{A+\cdots+A}_{p-1\text{ times}} = \FF_p^2. \]
\end{theorem}

In the case of a \emph{random} set $B$ of points in $\F_p^n$, we show that $n+2$ points suffice with high probability as $p$ grows, using a simple argument that studies covariances under random affine maps.

\begin{theorem}
    \label{thm:prob}
    Let $c\in (0,1)$ and $n\in \ZZ_{>0}$. For every sufficiently large prime $p$, a uniformly random set $B$ of $n+2$ points in $\F_p^n$ satisfies
    \[
        \underbrace{B+\cdots+B}_{\lceil cp\rceil\text{ times}}=\F_p^n,
    \]
    with probability $1-o_p(1)$.
\end{theorem}

One might ask about sufficient properties for a deterministic set of $n+2$ points in $\F_p^n$ to exhibit similar expansion behavior. We leave the characterization of such sets as a problem for future study.

\subsection{General algebraic bounds}
\cref{thm:main} is a special case of the following more general result, which gives a lower bound on the size of the sumset $A_1+\cdots+A_m$ as the relevant hypersurface degrees vary. Let $N(q,n,D)$ be the number of $n$-variable monomials of degree at most $D$ with individual degree at most $q-1$ in each variable.

\begin{theorem}
    \label{thm:main-finalp}
    Let $\FF=\FF_p$, and let $m, n_1,\dots,n_m$ be positive integers such that $n_1+\cdots+n_m\ge D$. If $A_1,\dots,A_m\subseteq \FF_p^n$, and for $1\le i\le m$, $A_i$ is not contained in the zero set of any polynomial of degree $\le n_i$, then 
    \[
        |A_1+\cdots+A_{m}| \ge N(p,n,D).
    \]
\end{theorem}

Note that we have $N(q,n,D)\le \binom{n+D}{D}$ for all $q,n,D$, while for $D\le (q-1)n$ we have $N(q,n,D)\ge \left(1+\lfloor \frac{D}{n}\rfloor\right)^n$. In particular, $N(p,n,n(p-1))\ge p^n$, so \cref{thm:main} follows from \cref{thm:main-finalp} applied with $D=n(p-1)$.

The same arguments can also be generalized to yield analogous results over fields $\FF_q$ of nonprime order, albeit with an additional, more complicated condition. For simplicity, we will restrict ourselves to the question of a sumset expanding to the whole space in this setting.

\begin{theorem}
    \label{thm:main-q}
    Let $\FF=\FF_q$, where $q$ is a power of a prime $p$, and let $m, n_1,\dots,n_m$ be positive integers such that $n_1+\cdots+n_m\ge (q-1)n$. If $A_1,\dots,A_m\subseteq \FF_q^n$, and for $1\le i\le m$, $A_i$ does not lie in the zero set of any polynomial of degree $\le n_i$, then 
    \[
        A_1+\cdots+A_{m} = \FF_q^n,
    \]
    as long as there exist $\alpha^{(1)},\dots,\alpha^{(m)}\in \N^n$ such that $|\alpha^{(i)}|\le n_i$, $\sum_{i=1}^m \alpha^{(i)} = (q-1,\dots,q-1)$, and
    \begin{equation*}
        \label{eqn:q-cond} \tag{1}
        \binom{(q-1,\dots,q-1)}{\alpha^{(1)},\dots,\alpha^{(m)}}
        \neq 0\pmod p.
    \end{equation*}
\end{theorem}
Here $\binom{(q-1,\dots,q-1)}{\alpha^{(1)},\dots,\alpha^{(m)}}= \frac{((q-1)!)^n}{\prod_{i=1}^m \alpha^{(i)}!}$, where for $\alpha=(\alpha_1,\dots,\alpha_n)$, $\alpha!$ denotes $\prod_{i=1}^n \alpha_i!$. Recall that the number of times $p$ divides $k!$ for a positive integer $k$ is given by
\[
    v_p(k!)=\sum_{i=1}^\infty \left\lfloor \frac{k}{p^i}\right\rfloor=\frac{k-s_p(k)}{p-1},
\]
where $s_p(k)$ is the sum of the digits of $k$ in base $p$. This means that $v_p((p^\ell-1)!)=(p-1)v_p((\frac{p^\ell-1}{p-1})!)$ for all $\ell\ge 1$. Thus, when $q=p^\ell$, \eqref{eqn:q-cond} is satisfied for $m=p-1$, $\alpha^{(i)}=(\frac{q-1}{p-1},\dots,\frac{q-1}{p-1})$. That is, the conclusion of \cref{thm:main-q} holds when $m=p-1$, $n_i=\frac{q-1}{p-1}n$ for $1\le i\le p-1$. In particular, as discussed in more detail in \cref{sec:discuss}, applying \cref{thm:main-q} to $\FF_{q^n}^1$ recovers Tao's bound of $|A|\ge \frac{p^n-1}{p-1}+1$ in the original question from \cite{chapman2023}.

Our proofs use a version of the polynomial method based on so-called shift operators, developed in \cite{Luo2023}. In \cref{sec:prelims}, we introduce the key definitions and tools needed for these proofs. The proofs of \cref{thm:main-finalp} and \cref{thm:main-q} are found in \cref{sec:proofs}, followed by a discussion of the low-dimensional setting in \cref{sec:low-dim}. The random set bound, \cref{thm:prob}, is proven in \cref{sec:random}. Further discussion of our results and their implications, including comparisons with other known results, are found in \cref{sec:discuss}.

\section{Preliminaries}
\label{sec:prelims}

In this section, we introduce some definitions and notation adopted from \cite{Luo2023}, before proving a new lemma that will be useful on the linear algebra side of the arguments that follow. 

\subsection{Basic definitions and Hasse derivatives}
Let $\FF$ be a field. 
For integers $a\leq b$, let $[a,b]$ denote the set of integers between $a$ and $b$ inclusive. For elements $v_1, \dots,v_m$ of a vector space $V$, denote by $\langle v_1,\dots,v_m\rangle$ the linear span of these elements. 

Let $\NN$ denote the set of nonnegative integers. Whenever we consider an $n$-tuple $\alpha\in \NN^n$, let its components be given by $\alpha=(\alpha_1,\dots,\alpha_n)$. Define the \emph{weight} of $\alpha$ by $|\alpha|:=\sum_{i=1}^n \alpha_i$. For $\alpha,\beta\in \NN^n$, we say $\alpha\leq \beta$ if $\alpha_i\leq \beta_i$ for all $i\in [1,n]$. Let $\alpha!=\prod_{i=1}^n \alpha_i!$, and $\binom{\alpha}{\beta}=\prod_{i=1}^n \binom{\alpha_i}{\beta_i}$. 
For any $\alpha\in \NN^n$, let $X^\alpha=\prod_{i=1}^n X_i^{\alpha_i}$. For $f\in \FF[X_1,\dots,X_n]$, let $[X^\alpha]f$ denote the coefficient of $X^\alpha$ in $f$. 

The $\alpha$th \emph{Hasse derivative} of $f$ is defined by
\[
H^{(\alpha)}f(X)=[Z^\alpha]f(X+Z),
\]
that is, the coefficient of $Z^\alpha$ in $f(X+Z)$ when treated as a polynomial in $Z$. In particular, note that $H^{(\alpha)}x^\beta=\binom{\beta}{\alpha}x^{\beta-\alpha}$ for $\alpha,\beta\in \NN^n$. Note also that $H^{(\alpha)}H^{(\beta)}f(X)=\binom{\alpha+\beta}{\alpha} H^{(\alpha+\beta)}f(X) = H^{(\beta)}H^{(\alpha)}f(X)$, i.e. Hasse derivatives commute with each other as operators. 

\subsection{Shift operators}

For $h\in \FF^n$, we define the linear operator $T^h$ on the space of polynomials $P_n=\FF[X_1,\dots,X_n]$ by
\[
T^h(f)(X)=f(X+h).
\]
We call these the \emph{shift operators}. 
From the definition, it is clear that $T^a T^b=T^{a+b}$ for all $a,b\in \FF^n$, and that
\begin{equation}
    \label{shift-to-hasse}
    T^h = \sum_{\alpha\in \NN^n} h^\alpha H^{(\alpha)}.
\end{equation}

Given a set $A\subseteq \FF^n$, let $\Lambda_A$ denote the space of linear combinations of $\{T^a\}_{a\in A}$, as operators on $\FF[X_1,\dots,X_n]$. 
Applying \eqref{shift-to-hasse}, each such linear combination $\ell$ can be written as a linear combination of (Hasse) derivatives. In analogy with coefficients of polynomials, we can define $[H^{(\alpha)}]\ell$ as the coefficient of $H^{(\alpha)}$ in $\ell$ when expressed in this ``derivative expansion''. Define the degree $\deg (\ell)$ to be the minimal weight over all $\alpha\in \NN^n$ such that $[H^{(\alpha)}]\ell\neq 0$. If such an $\alpha$ does not exist, i.e.~if $\ell$ is identically zero, we write $\deg(\ell)=\infty$. Write $\ell_{(d)}$ for the degree $d$ component of $\ell$ in such a representation; that is,
\[
\ell_{(d)}=\sum_{\alpha:\:|\alpha|=d}([H^{(\alpha)}]\ell) H^{(\alpha)}.
\]
In many cases, it will be helpful to focus on the ``leading component'' $\ell_{(\deg(\ell))}$. Let $\delta(\ell)$ denote this leading component. For each $d\geq 0$, define $\Delta_A^d=\{\ell_{(d)}:\: \ell\in \Lambda_A,\: \deg(\ell)\geq d \}$, and let $\Delta_A=\bigcup_{d\geq 0} \Delta_A^d$. Thus each $\Delta_A^d$ is a space of linear operators on $\FF[X_1,\dots,X_n]$, and $\Delta_A$, the set of all possible leading terms, is a union of a chain of these spaces. Let $\deg(A)$ denote the largest $d$ such that $\Delta_A^d\neq \{0\}$.

Some of the important basic properties of shift operators that we will use in our proofs are collected in the following statement.
\begin{lemma}
    \label{lem:shift-properties}
    Let $A,B\subseteq \FF^n$.
    \begin{enumerate}
    \item[(a)] (Linear independence) The set of shift operators $\{T^a\}_{a\in A}$ is linearly independent. In particular, $\sum_{d\ge 0}\dim(\Delta_A^d)=\dim(\Lambda_A)=|A|$. 
    \item[(b)] (Additivity) $\Lambda_A\cdot \Lambda_B\subseteq \Lambda_{A+B}$, and therefore $\Delta_A\cdot\Delta_B\subseteq \Delta_{A+B}$.
    \item[(c)] (Unique max degree) $\deg(A)\le n(|\FF|-1)$, with equality if and only if $A=\FF^n$.
    \item[(d)] (Reduction) If $\sum_{\alpha\in \NN^n} c_\alpha H^{(\alpha)}\in \Lambda_A$ for some constants $c_\alpha\in \FF$, then for each $i\in [1,n]$, we have $\sum_{\alpha\in \NN^n} c_{\alpha+e_i} H^{(\alpha)}\in \Lambda_A$, where $e_i$ is the $n$-tuple with a $1$ in the $i$th coordinate and $0$s everywhere else. 
\end{enumerate}
\end{lemma}
Part~(a) of \cref{lem:shift-properties} follows from \cite[Lemma~5.3]{Luo2023}, part~(b) from the proof of \cite[Lemma~5.5]{Luo2023}, part~(c) from \cite[Proposition~5.4]{Luo2023}, and part~(d) from \cite[Lemma~5.1]{Luo2023}.

\cref{lem:shift-properties}(a) tells us that the shift operators corresponding to a large set of points $A$ in $\FF^n$ must span many dimensions worth of lowest degree terms in their derivative expansions. In \cref{sec:matrices}, we will show that under certain conditions, we can say much more about which such lowest degree terms are attained.

\subsection{Rank-degree lemma}
\label{sec:matrices}
The goal in this section is to prove the following lemma.

\begin{lemma}\label{lem-generalPositionHasse}
If $A\subset F^n$ is not contained in any hypersurface of degree at most $d$, then $\Delta_A$ contains every Hasse derivative $H^{(\alpha)}$ of order at most $d$.
\end{lemma}

\begin{proof}[Proof of Lemma~\ref{lem-generalPositionHasse}]
Recall that the coefficient of $H^{(\alpha)}$ in $T^h$ is $h^\alpha$. We define the evaluation matrix $M:=\Eval(A,[0,d])$ to be a matrix with rows labeled by points in $A\subset \F^n$ and columns labeled by $\{\alpha \in \N^n\mid |\alpha|\in [0,d]\}$. The $(a,\alpha)$-th entry in $M$ is $a^\alpha$, and the $a$th row of $M$ encodes the coefficients of the degree $\le d$ part of the Hasse derivative expansion of $T^a$.

We claim that $M$ has full column rank. Indeed, let $v=(c_\alpha)_{|\alpha|\le d}$ be a nonzero point in the column kernel of $M$. Then for all $a\in A$, we have $\sum_{|\alpha|\le d}c_\alpha a^\alpha=0$, yielding a polynomial $p(x)=\sum_{|\alpha|\le d}c_\alpha x^\alpha$ of degree at most $d$ vanishing on all of $A$, which is a contradiction.

As a result, the row space of $M$ spans all vectors in $\FF^{|A|}$. In particular, for each monomial $\alpha$ wth $|\alpha|\le d$, there exists a linear combination $w$ of the rows that equals the indicator vector for $\alpha$, i.e. $w_\alpha=1$ and $w_{\alpha'}=0$ for all $\alpha'\neq \alpha$ with $|\alpha'|\le d$. But $w$ encodes the coefficients of the degree $\le d$ part of the Hasse derivative expansion of some $\ell \in \Lambda_A$. Thus, $H^{(\alpha)}\in \Delta_A$ for each $\alpha$ of weight at most $d$, as desired.
\end{proof}

\section{Proof of main result}
\label{sec:proofs}
We now give a proof for \cref{thm:main-finalp}, the general version of our main result. Recall that $N(q,n,D)$ denotes the number of monomials of degree at most $D$ with individual degree at most $q-1$ in each variable.

\begin{proof}[Proof of \cref{thm:main-finalp}]
Let $\FF=\FF_p$.
Since $n_1+\cdots+n_m\ge D$, for any $z\in \NN^n$ of weight at most $D$ and individual weight at most $p-1$, we can fix some choice of $\alpha^{(1)},\dots,\alpha^{(m)}\in \NN^n$ such that $\alpha^{(i)}$ has weight at most $n_i$ for each $i\in [1,m]$, and $\alpha^{(1)}+\cdots+\alpha^{(m)}=z$. 
We have that $A_1,\dots,A_{m}\subseteq \FF^n$, and $A_i$ is not contained within any hypersurface of degree at most $n_i$. Let $S=A_1+\cdots+A_{m}$. 
By \cref{lem:shift-properties}(b), we have
\[
\Delta_{A_1}\cdot \cdots \cdot \Delta_{A_{m}}\subseteq \Delta_S,
\]
while by Lemma~\ref{lem-generalPositionHasse}, $\Delta_{A_i}$ contains every Hasse derivative of order at most $n_i$. In particular, we have that $H^{(\alpha^{(i)})}\in \Delta_{A_i}$ for each $i$, so that 
\[ \Delta_S \ni (H^{(\alpha^{(1)})})\cdots(H^{(\alpha^{(m)})})=\binom{z}{\alpha^{(1)},\dots,\alpha^{(m)}} H^{(z)}\neq 0,\]
since the binomial coefficient $\binom{z}{\alpha^{(1)},\dots,\alpha^{(m)}}=\frac{z!}{\prod_{i=1}^m(\alpha^{(i)}!)}$ is nonzero mod $p$ for our choices of the $\alpha^{(i)}$. Since the number of choices of $z\in \N^n$ with weight at most $D$ and individual weight at most $p-1$ is $N(p,n,D)$, we have $|S|\ge N(p,n,D)$ as desired. 
\end{proof}

The proof of \cref{thm:main-q} follows from a similar argument applied to a general finite field $\F_q$. In this case, the step of the proof requiring a certain multinomial coefficient to be nonzero introduces an extra condition.

\begin{proof}[Proof of Theorem~\ref{thm:main-q}]

We have that $A_1,\dots,A_{m}\subseteq \FF_q^n$, and $A_i$ is not contained within any hypersurface of degree at most $n_i$. Let $\alpha^{(1)},\dots,\alpha^{(m)}\in \N^n$ be as described in the last condition in the theorem statement. Let $S=A_1+\cdots+A_{m}$.

As in the proof of \cref{thm:main-finalp}, we obtain by \cref{lem:shift-properties}(b) that
\[
\Delta_{A_1}\cdot \cdots \cdot \Delta_{A_{m}}\subseteq \Delta_S,
\]
while by Lemma~\ref{lem-generalPositionHasse}, $\Delta_{A_i}$ contains every Hasse derivative of order at most $n_i$. In particular, we have that $H^{(\alpha^{(i)})}\in \Delta_{A_i}$ for each $i$, so that 
\[ \Delta_S \ni (H^{(\alpha^{(1)})})\cdots(H^{(\alpha^{(m)})})=\binom{(q-1,\hdots,q-1)}{\alpha^{(1)},\dots,\alpha^{(m)}} H^{(q-1,\dots,q-1)}\neq 0,\]
by the assumption that $\binom{(q-1,\dots,q-1)}{\alpha^{(1)},\dots,\alpha^{(m)}}\neq 0 \pmod p$. Then $\deg(A)\ge |(q-1,\dots,q-1)|=n(q-1)$, so by Lemma~\ref{lem:shift-properties}(c), we have $A=\F_q^n$ as desired.
\end{proof}

\section{Two dimensions}
\label{sec:low-dim}
The condition in Theorem~\ref{thm:main} about sets not lying in hypersurfaces of low degree are simple and general, but not, it seems, fully optimized. Intuitively, while being contained in the union of a small number of hyperplanes should hinder additive expansion, it does not seem that being correlated with a nonlinear polynomial of low degree should inherently have the same effect. In the proof of \cref{thm:2d}, we explore this intuition in the symmetric case by finding one way in which the notion of a sufficiently generic set can be further relaxed in the two-dimensional setting.

\begin{proof}[Proof of \cref{thm:2d}]
    We start by replacing $A$ with a subset consisting of $4$ points, no three of which lie on a line. Since affine transformations on $A$ do not affect the sizes of its iterated sumsets, we can assume without loss of generality that $(0,0),(0,1),(1,0)\in A$. Let the fourth point be $(a,b)$. By considering the lowest degree terms in the Hasse derivative expansions of $T^{(0,0)}$, $T^{(1,0)}-T^{(0,0)}$, and $T^{(0,1)}-T^{(0,0)}$, we can already obtain $\Delta_A^{0}=\langle 1\rangle $ and $\Delta_A^{1}=\langle H^{(1,0)},H^{(0,1)}\rangle$. By \cref{lem:shift-properties}(d), since $\sum_{d\ge 0}\dim(\Delta_A^d)=|A|=4$, the last remaining dimension worth of lowest degree terms must come from degree $2$. Let $\ell\in \Lambda_A$ satisfy $\deg(\ell)=2$, with lowest degree component $\ell_{(2)}$. Proceeding as in the proofs of \cref{thm:main-finalp} and \cref{thm:main-q}, it suffices to show that $\ell_{(2)}^{p-1} \neq 0$ as long as $(a,b)$ does not lie in a line with two of the other points of $A$.

    Expanding out
    \[
        T^{(x,y)}=\sum_{i,j\ge 0}x^i y^j H^{(i,j)},
    \]
    the unique (up to scaling) linear combination of $\{T^h\}_{h\in A}$ giving cancellation in the three terms of degree $<2$ is
    \[
        \ell:=T^{(a,b)}-aT^{(1,0)}-bT^{(0,1)}+(a+b-1)T^{(0,0)},
    \]
    which has lowest degree component
    \[
    \ell_{(2)}=(a^2-a) H^{(2,0)}+ab H^{(1,1)}+(b^2-b)H^{(0,2)}=\frac{a^2-a}{2}(H^{(1,0)})^2 + abH^{(1,0)}H^{(0,1)}+\frac{b^2-b}{2}(H^{(0,1)})^2.
    \]
    Viewing this last expression as a quadratic in $H^{(1,0)}$ and $H^{(0,1)}$, as long as its discriminant is nonzero, we can write
    \[
        \ell_{(2)}=c_1 (H^{(1,0)}+c_2 H^{(0,1)})^2+c_3 (H^{(0,1)})^2,
    \]
    where $c_1,c_3\neq 0$. Noting that for any $x,y\in \FF_p$ we have $(xH^{(1,0)}+yH^{(0,1)})^p=x (H^{(1,0)})^p+y(H^{(0,1)})^p=0$, we then have
    \[
        \ell_{(2)}^{p-1}=\binom{p-1}{\frac{p-1}{2}} c_1^{\frac{p-1}{2}} (H^{(1,0)}+c_2 H^{(0,1)})^{p-1} c_3^{\frac{p-1}{2}} (H^{(0,1)})^{p-1}=\binom{p-1}{\frac{p-1}{2}}(c_1 c_3)^{\frac{p-1}{2}} (p-1)! H^{(p-1,p-1)}\neq 0.
    \]
    Thus it suffices to verify that the relevant discriminant is nonzero. Said discriminant evaluates to
    \[
    (ab)^2-(a^2-a)(b^2-b)=ab(a+b-1),
    \]
    which is zero if and only if $(a,b)$ lies on one of the three lines formed by pairs of points among $(0,0)$, $(1,0)$, and $(0,1)$. This proves the desired claim.
\end{proof}

One can attempt to prove similar results for any fixed number of dimensions $n$. For a set $A$ of a fixed size (say, $\binom{2n-1}{n}+1$), it suffices to find a general expression for a linear combination $\ell \in \Lambda_A$ with $\deg (\ell) \ge n$, then analyze the conditions on $A$ under which one can guarantee that $\ell_{(n)}^{p-1}\neq 0$. While nothing as well understood as the discriminant is likely to arise in such an analysis for $n>2$, there is nevertheless room for interesting discoveries in this direction.

\section{Bounds for random point sets}\label{sec:random}

Our goal in this section is to prove Theorem~\ref{thm:prob} by studying a particular family of hash functions, which we define below.

\begin{definition}
For a prime $p$ and integer $d\in [p-1]$, we equitably partition $\{0,1,\hdots,p-1\}$ into $d$ intervals $I_1,\dots,I_{d}$ defined by $I_i=\{\lfloor (i-1)p/d\rfloor,\hdots,\lfloor ip/d\rfloor-1\}$ for $1\le i\le d$. We can now split $\F_p^n$ into $d^n$ rectangles  $R_k$ labeled by $k\in [d]^n$. We call this the \emph{$d$-cube partition} of $\F_p^n$.
\end{definition}
Note that $||I_i|-p/d|<1$ for all $i$, meaning $|I_i|\in \{\lfloor p/d\rfloor,\lceil p/d\rceil\}$, so $|I_i|,|I_j|$ differ by at most $1$ for all $i,j$.

\begin{definition}
For $b\in \F_p^n$, $A\in \operatorname{GL}(\F_p^n)$, and $d\in [p-1]$, we define a map $f_{A,b,d}$ from $\F_p^n$ to $[ d]^n$ by mapping $x$ to the label of the rectangle of the $d$-cube partition that $Ax+b$ is in. That is, for $j\in [n]$, letting $y_j$ denote the $j$th coordinate of a point $y$, we define $f_{A,b,d}(x)_j=i$ if $(Ax+b)_j\in I_i$.
\end{definition}
We first show that for a fixed $d$, the family of maps $\{f_{A,b,d}\}_{b\in \F_p^n,\, A\in \operatorname{GL}(\F_p^n)}$ is close to pairwise independent. 
For $t\in \F_p^n$ and $k\in [d]^n$, let $X_{t,k}$ be the indicator variable for the event that $f_{A,b,d}(t)=k$.
\begin{prop}
\label{prop:covar}
For $d\in [p-1]$ and distinct $x, y \in \F_p^n$, when $b\in \F_p^n$ and $A \in \operatorname{GL}(\F_p^n)$ are chosen uniformly at random, we have
\[
    \Cov(X_{x,k},X_{y,\ell})\le \left(\frac{2}{dp^2}\right)^n,
\]
for all $k,\ell\in [d]^n$.
\end{prop}
\begin{proof}
By the construction of $f_{A,b,d}$, it suffices to consider the correlation between $Ax+b,Ay+b$ for $x\neq y$. For any $s\neq t\in \F_p^n$, we have
\begin{align*}
    \Pr[Ax+b=s \, \wedge \, Ay+b=t]&=\Pr[A(y-x)=t-s \, \wedge \, b=s-Ax ] \\
    &=\Pr[A(y-x)=t-s]\Pr[b=s-Ax]=\frac{1}{p^n-1}\frac{1}{p^n}
    ,
\end{align*}
while for $s=t$, we have $\Pr[Ax+b=s \, \wedge \, Ay+b=t]=0$ for $x\neq y$.

Note that for any $x\in \F_p^n$ and $k\in [d]^n$, we have $\Pr[X_{x,k}=1]=\Pr[f_{A,b,d}(x)=k]=\sum_{s\in R_k}\Pr[Ax+b=s]=|R_k|\frac{1}{p^n}$. Then
\[
\Pr[X_{x,k}X_{y,\ell}=1]=\sum_{s\in R_k,\, t\in R_\ell}\Pr[Ax+b=s\,\wedge\, Ay+b=t]=\begin{cases}
    |R_k||R_\ell|\frac{1}{p^n-1}\frac{1}{p^n} & \text{ if }k\neq \ell, \\
    (|R_k||R_\ell|-1)\frac{1}{p^n-1}\frac{1}{p^n} & \text{ if }k= \ell.
\end{cases}
\]
Since $(\lfloor \frac{p}{d}\rfloor )^n<|R_k| < (\lceil \frac{p}{d}\rceil )^n$, for $x\neq y$ and $k\neq \ell$ we have
\[
    \Cov(X_{x,k},X_{y,\ell})=|R_k||R_\ell|\left(\frac{1}{p^n-1}\frac{1}{p^n}-\frac{1}{p^{2n}}\right)=\frac{|R_k||R_\ell|}{(p^n-1)p^{2n}}\le \left(\frac{2}{dp^2}\right)^n,
\]
while for $x\neq y$ and $k=\ell$ we have
\[
    \Cov(X_{x,k},X_{y,\ell})=
    (|R_k||R_\ell|-1)\frac{1}{p^n-1}\frac{1}{p^n}-|R_k||R_\ell|\frac{1}{p^{2n}} = (|R_k||R_\ell|-1)(\frac{1}{p^n-1}-\frac{1}{p^n})-\frac{1}{p^n} < \left(\frac{2}{dp^2}\right)^n,
\]
as claimed.
\end{proof}

Next, we show that for fixed $d$, a randomly chosen map $f_{A,b,d}$ will map any set of $\Omega(p)$ points surjectively onto $[d]^n$ with high probability. This immediately follows from a mild generalization of the Leftover Hash Lemma (which in fact will show any large enough set will be `equally distributed' in $\ell_1$ distance). We give a direct proof for the statement we need using Chebyshev's inequality.
\begin{lemma}
\label{lem:hash}
Let $p$ be a prime, $d\in [p-1]$, $c\in (0,1)$, and $S\subseteq \F_p^n$ with $|S|\ge cp$. For a uniformly random choice of $b\in \F_p^n$ and $A\in\operatorname{GL}(\F_p^n)$, we have that $f_{A,b,d}$ surjects $S$ onto $[d]^n$ with probability at least $1-(9d^2)^{n}/cp$.
\end{lemma}
\begin{proof}
Fix any label $k\in [d]^n$. For $t\in \F_p^n$, recall that we defined $X_{t,k}$ to be the indicator variable for the event that $f_{A,b,d}(t)=k$. Also recall from the proof of \cref{prop:covar} that $\E[X_{t,k}]=\Pr[X_{t,k}=1]=|R_k|\frac{1}{p^n}$. By Chebyshev's inequality, using the bound on covariances from \cref{prop:covar}, we have
\begin{align*}
    \Pr\left[\sum\limits_{t\in S} X_{t,k} = 0\right]&\le \frac{\operatorname{Var}\left(\sum\limits_{t\in S} X_{t,k}\right)}{\left(\E[\sum\limits_{t\in S} X_{t,k}]\right)^2} =\frac{|S|(\E[X_{t,k}]-\E[X_{t,k}]^2)+\sum_{t,t'\in S:\: t\neq t'}\Cov(X_{t,k},X_{t',k})}{|S|^2\E[X_{t,k}]^2} \\
    &\le \frac{|S||R_k|\frac{1}{p^n}+(|S|^2-|S|)(\frac{2}{dp^2})^n}{(|S||R_k|\frac{1}{p^n})^2} \le \frac{(2/d)^n+|S|(\frac{2}{dp^2})^n}{|S|/(2d)^{2n}} \\
    &\le \frac{(8d)^n}{cp}+\frac{1}{(2d^3 p^2)^n}\le \frac{(9d)^n}{cp}.
\end{align*}
Taking a union bound over all $k\in [d]^n$, we see that the probability of $f_{A,b,d}$ not being surjective is upper bounded by $\frac{(9d^2)^{n}}{cp}$, as desired.
\end{proof}

\begin{proof}[Proof of \cref{thm:prob}]
    Let $B=\{s_0,\dots,s_{n+1}\}$. First, the probability that uniformly and independently chosen points $s_0,\dots,s_{n}\in \FF_p^n$ affinely span the whole space is at least
    \[
        (1-p^{-n})(1-p^{-{n-1}})\dots (1-p^{-1})=1-o_p(1).
    \]
    We condition on this high-probability event holding, so that $(s_0,\dots,s_n)$ is a uniformly random tuple of affinely independent points in $\F_p^n$ under this conditioning. Then there is a unique choice of $b\in \FF_p^n$ and $A\in \operatorname{GL}(\F_p^n)$ such that the map $z\mapsto Az+b$ sends $(s_0,\dots,s_n)$ to $(0,e_1,\dots,e_n)$, where $e_1,\dots,e_n$ form the coordinate basis for $\FF_p^n$, and $b,A,s_{n+1}$ are uniformly random and independent under this conditioning. Let $B_0=\{0,e_1,\dots,e_n\}$.

    Fix $s_{n+1}=x$, and let $S=\{x,\dots,\lfloor\frac{1}{2}cp\rfloor x\}$. Let $B'$ be the image of $B$ under the map $z\mapsto Az+b$, so $B'=B_0\cup \{Ax+b\}$. Since this map is invertible, to show that $\underbrace{B+\cdots+B}_{\lceil cp\rceil\text{ times}}=\F_p^n$, it suffices to show that $\underbrace{B'+\cdots+B'}_{\lceil cp\rceil\text{ times}}=\F_p^n$ with high probability. But
    \[
    \underbrace{B'+\cdots+B'}_{\lceil cp\rceil\text{ times}}\supseteq \underbrace{B_0+\cdots+B_0}_{\lfloor cp/2\rfloor \text{ times}}+(AS+b).
    \]
    The first sum on the right hand side includes all points in the box $[0,\lfloor \frac{1}{n} (\frac{cp}{2}-1)\rfloor ]^n\supseteq [0,2\lfloor \frac{cp}{6n}\rfloor]^n$ for $p$ sufficiently large in terms of $c$ and $n$. In particular, letting $d=\frac{7n}{c}$, for every $y\in \F_p^n$, there is a rectangle $R_k$ in the $d$-cube partition of $\F_p^n$ such that $y-[0,2\lfloor \frac{cp}{6n}\rfloor]^n\supseteq R_k$, i.e. $y\in z+[0,2\lfloor \frac{cp}{6n}\rfloor]^n$ for every $z\in R_k$. Now, since $A,b$ are still uniformly random and independent, by \cref{lem:hash}, with probability at least $1-(9d^2)^n/cp$, 
    $AS+b=\{Ax+b,\dots,A\left\lfloor\frac{1}{2}cp\right\rfloor x+b\}$ contains a point from each rectangle $R_k$, which by the above implies that $[0,2\lfloor \frac{cp}{6n}\rfloor]^n+(AS+b)=\F_p^n$. Thus when $p$ is sufficiently large in terms of $n$ and $c$, we indeed have that $\underbrace{B'+\cdots+B'}_{\lceil cp\rceil\text{ times}}=\F_p^n$, and thus $\underbrace{B+\cdots+B}_{\lceil cp\rceil\text{ times}}=\F_p^n$ as desired.
\end{proof}
Note that this proof does not make use of the structure of $S$ (a long arithmetic progression), only its size; this and several other parts of the argument that are quite loose suggest that there may be room for improvement in the number of summands required to expand to the whole space in \cref{thm:prob}.

\section{Concluding remarks}
\label{sec:discuss}

\subsection{Comparison of \cref{thm:main} with \cref{thm:main-q}}
For any prime $p$, integer $n$, and $\ell\mid n$, one can canonically identify the additive group structure of the space $\FF_p^n$ with the space $\FF_{p^\ell}^{n/\ell}$ via a group isomorphism $\varphi_\ell$ (which is in fact an isomorphism of $\FF_p$-vector spaces). A set $A\subseteq \FF_p^n$ can thus be identified with a subset $\varphi_\ell(A)$ of $\FF_{p^\ell}^{n/\ell}$. 

This identification allows us to use \cref{thm:main-q} to obtain a family of conditions for sumset expansion whenever we would normally apply \cref{thm:main}. For example, consider the sumset $\underbrace{A+\cdots+A}_{p-1\text{ times}}$. \cref{thm:main-symm} gives a condition on $A$ under which this sumset is guaranteed to be the whole space $\FF_p^n$: it suffices to know that $A$ is not contained in the zero set of any polynomial of degree $\le n$. However, since $\underbrace{\varphi_\ell(A)+\cdots+\varphi_\ell(A)}_{p-1\text{ times}}=\varphi_\ell(\underbrace{A+\cdots+A}_{p-1\text{ times}})$, we can apply \cref{thm:main-q} to $\varphi_\ell(A)\subseteq \FF_{p^\ell}^{n/\ell}$ for each $\ell\mid n$ to obtain alternate sufficient conditions for reaching this conclusion. Namely, by setting $n_i=\frac{p^\ell-1}{p-1}n$ for $1\le i\le p-1$, \cref{thm:main-q} yields that it suffices to have \textbf{some} $\ell\mid n$ such that $\varphi_\ell(A)$ does not lie in the zero set of a degree $\le \frac{p^\ell-1}{p-1}n$ polynomial over $\FF_{p^\ell}$.

Setting $\ell=1$ recovers the condition given by \cref{thm:main-symm}. Alternate conditions from the $\ell\neq 1$ cases are not as easy to work with (or as permissive) in general. For example, when $n=4$, \cref{thm:main-symm} shows that every set $A\subset \F_p^4$ that is not contained in the zero set of any polynomial of degree $\le 4$ satisfies $\underbrace{A^{(p-1)}+\cdots+A^{(p-1)}}_{p-1\text{ times}}=\F_p^4$, a condition that any set of $\binom{8}{4}+1=71$ suitably generic points satisfies. Attempting to apply \cref{thm:main-q} to $\FF_{p^2}^2$ would require $A$ to avoid all polynomials of degree at most $2(p-1)$, therefore requiring $|A|=\Omega(p^2)$. 

However, there is one other setting worth noting. When $\ell = n$, we have $\varphi_\ell(A)\subseteq \FF_{p^n}^1$, so the condition that $\varphi_\ell(A)$ does not lie in the zero set of any polynomial of degree $\le \frac{p^n-1}{p-1}$ is equivalent to the condition that $|A|>\frac{p^n-1}{p-1}$. This indeed recovers Tao's bound of $|A|\ge \frac{p^n-1}{p-1}+1$ for the original question from \cite{chapman2023} -- something that \cref{thm:main-symm} does not directly yield. 

\subsection{Affine Bases and the Erd\H{o}s-Ginzburg-Ziv problem}

One natural source of motivation for questions about iterated sumset expansion is the \emph{Erd\H{o}s-Ginzburg-Ziv problem}: What is the smallest integer $s$ such that every sequence of $s$ elements of $\FF_p^n$ contains a subsequence of $p$ elements with zero sum? This constant $s=s(\FF_p^n)$ is known as the \emph{Erd\H{o}s-Ginzburg-Ziv constant} of $\FF_p^n$. Clearly, if a sequence (viewed as a multiset) can be partitioned into sets $A_1,\dots,A_{p}$ such that $A_1+\cdots+A_{p}=\FF_p^n$, then in particular we will have $0\in A_1+\cdots+A_p$. Thus, understanding the structural properties that determine whether or not such an iterated sumset expands to the whole space is an important part of obtaining bounds on $s(\FF_p^n)$. 

This is one of the main ideas in \cite{alondubiner1995}, where Alon and Dubiner showed that $s(\FF_p^n)\le C_n p$, where $C_n\le (cn \log n)^n$ for some absolute constant $c>0$. Similar ideas also show up in more recent work on the Erd\H{o}s-Ginzburg-Ziv problem, including Zakharov's proof in \cite{zakharov2020convexEGZ} that $s(\FF_p^n)\le 4^n p$ for fixed $n$ and sufficiently large $p$.

The following proposition is one of the key steps in \cite{alondubiner1995}, where a proof is given using the Pl\"unnecke-Ruzsa inequality. Recall that an \emph{affine basis} of $\FF_p^n$ is a set of $n+1$ vectors that is affinely independent (i.e. not contained in a hyperplane).

\begin{prop}{\cite[Proposition 2.1]{alondubiner1995}}
    Let $x\le p/4n$ be a power of $2$, and let $A_1,\dots,A_m$ be $m$ affine bases of $\FF_p^n$, where $m=4xn$. Then
    \[
        |A_1+\cdots+A_m|\ge x^n.
    \]
\end{prop}

Here we give a quick proof of a slightly stronger version of this proposition using \cref{thm:main-finalp}.

\begin{prop}
    Let $A_1,\dots,A_m$ be $m\ge 1$ affine bases of $\FF_p^n$. Then
    \[
        |A_1+\cdots+A_m|\ge \min(p^n,\left(1+\left\lfloor \frac{m}{n}\right\rfloor\right)^n).
    \]
\end{prop}

\begin{proof}
    By the definition of an affine basis, each of the sets $A_i$ is not contained in any hyperplane, i.e. the zero set of any polynomial of degree $\le 1$. Applying \cref{thm:main-finalp} with $n_1=\cdots=n_m=1$, $D=m$ gives
    \[
        |A_1+\cdots+A_m|\ge N(p,n,m),
    \]
    where $N(p,n,m)$ is the number of $n$-variable monomials of degree at most $m$ with individual degree at most $p-1$ in each variable. As noted in \cref{sec:intro}, if $m\le (p-1)n$, we have $N(p,n,m)\ge (1+\lfloor \frac{m}{n}\rfloor)^n$ (by restricting the count to monomials with degree between $0$ and $\lfloor \frac{m}{n}\rfloor$ in each variable). If $m>(p-1)n$, then we have $N(p,n,m)\ge N(p,n,(p-1)n)\ge p^n$. Thus in either case, we have the desired lower bound.
\end{proof}
Note that this proof allows us to drop the condition that $x=\frac{m}{4n}$ is a power of $2$ (or in fact an integer), relax the restriction on $m$ (essentially allowing all $m\le (p-1)n$),  and improve the lower bound by a factor of $4^n$.

\bibliographystyle{acm}
\bibliography{main}

\end{document}